\documentclass{amsart}

\RequirePackage[OT1]{fontenc}
\usepackage{amssymb}
\usepackage[stretch=10,shrink=10]{microtype}
\usepackage[showframe=false, margin = 1.5 in]{geometry}
\usepackage{mathtools}
\usepackage[shortlabels]{enumitem}
\usepackage{theoremref}
\usepackage{tikz}
\usepackage{comment}
\usepackage{mathrsfs}
\def\bZ{\mathbb{Z}}
\newcommand{\ind}[1]{\mathbf{1}{\{ #1 \}}}

\newtheorem{thm}{Theorem}
\numberwithin{thm}{section}
\newtheorem{lemma}[thm]{Lemma}
\newtheorem{prop}[thm]{Proposition}

\theoremstyle{remark}
\newtheorem{remark}[thm]{Remark}
\theoremstyle{definition}
\newtheorem{define}[thm]{Definition}

\title{\bf Stretched exponential decay for subcritical parking times on $\mathbb{Z}^d$}

\author{Michael Damron}
\author{Hanbaek Lyu}
\author{David Sivakoff}

\date{}

\begin{document}

\maketitle

\begin{abstract}In the parking model on $\mathbb{Z}^d$, each vertex is initially occupied by a car (with probability $p$) or by a vacant parking spot (with probability $1-p$). Cars perform independent random walks and when they enter a vacant spot, they park there, thereby rendering the spot occupied. Cars visiting occupied spots simply keep driving (continuing their random walk). It is known that $p=1/2$ is a critical value in the sense that the origin is a.s.~visited by finitely many distinct cars when $p<1/2$, and by infinitely many distinct cars when $p\geq 1/2$. Furthermore, any given car a.s.~eventually parks for $p \leq 1/2$ and with positive probability does not park for $p > 1/2$. We study the subcritical phase and prove that the tail of the parking time $\tau$ of the car initially at the origin obeys the bounds
\[
\exp\left( - C_1 t^{\frac{d}{d+2}}\right) \leq \mathbb{P}_p(\tau > t) \leq \exp\left( - c_2 t^{\frac{d}{d+2}}\right)
\]
for $p>0$ sufficiently small. For $d=1$, we prove these inequalities for all $p \in [0,1/2)$. This result presents an asymmetry with the supercritical phase ($p>1/2$), where methods of Bramson--Lebowitz imply that for $d=1$ the corresponding tail of the parking time of the parking spot of the origin decays like $e^{-c\sqrt{t}}$. Our exponent $d/(d+2)$ also differs from those previously obtained in the case of moving obstacles.
\end{abstract}

\section{Statement of result}
In this note, we study the tail of the distribution of the parking time of cars in the subcritical phase for the parking model on $\mathbb{Z}^d$.  The model was introduced in \cite{DGJLS} and studied further in \cite{PRS} (a similar continuous-time particle system was introduced and studied earlier in \cite{CRS14, CRS18}), and is roughly defined as follows. Each vertex is initially assigned either a car (with probability $p \in [0,1]$) or a vacant parking spot (with probability $1-p$), independently between vertices. Cars perform simple symmetric discrete-time random walks, and when a car visits a vacant spot, it parks there. (If multiple cars visit the same vacant spot, the tie is broken using independent uniform $[0,1]$ random variables.) If a car visits an occupied spot, it does not park, and moves at the next timestep. (We refer the reader to~\cite{DGJLS} for a more precise definition of the process. Our probability space is defined in Section~\ref{sec:lower bound}.)

This parking model is a type of annihilating particle system, and can be described with the short-hand rule $A+B \to 0$, where cars are thought of as $A$ particles, spots are $B$ particles, and $A$ and $B$ particles annihilate each other upon collision. There is a large literature of such annihilating systems, and we only mention a few references including some for annihilating random walk \cite{A,EN,G} and for two-type annihilating random walk \cite{BL}. From the other side, the parking model can be viewed as a generalization of the study of parking functions, which are certain types of hashing functions resulting from particles parking in order on a line (see \cite{KW} and generalizations to trees in \cite{GP}).

In \cite{CRS14, CRS18, DGJLS} it was shown that there is a phase transition in $p$: a.s.~the origin is visited by finitely many distinct cars for $p<1/2$ (subcritical phase) and infinitely many for $p\geq 1/2$ (critical and supercritical phases). Furthermore, the parking time for the car of the origin (if it exists) is finite a.s.~for $p \leq 1/2$, and infinite with positive probability when $p > 1/2$. Our main result addresses the subcritical phase and determines the decay rate of the distribution of the parking time of cars. To state it precisely, we use the following notation: for each $v\in \mathbb{Z}^d$, the parking time $\tau^{(v)}$ is  
\begin{align}\label{def:liftspan}
\tau^{(v)} := \sum_{s=1}^{\infty} \ind{ \text{a car starts at $v$ and is unparked at time $s-1$}}.
\end{align} 
When $v=0$ we simply write $\tau$ for $\tau^{(0)}$. Note that if a car does not start at $v$, the convention taken is that $\tau^{(v)}=0$. We also write $\mathbb{P}_p$ for the probability measure associated to the parking model with parameter $p$.

We will prove here:
\begin{thm}\label{thm: main_thm}
For any $p > 0$, there exists $C_1>0$ such that
\[
\mathbb{P}_p(\tau > t) \geq \exp\left( - C_1 t^{\frac{d}{d+2}}\right) \text{ for all integers } t \geq 1.
\]
For 
\[
p_0 =
\begin{cases}
\frac{1}{2} & \quad \text{if } d =1 \\
\frac{1}{2} \left[ 1- \sqrt{1-(de)^{-2}}\right] & \quad \text{if } d \geq 2,
\end{cases}
\] 
if $p \in [0,p_0)$, then for some $c_2 = c_2(p)>0$,
\[
\mathbb{P}_p(\tau > t) \leq p\exp\left( -c_2 t^{\frac{d}{d+2}}\right) \text{ for all integers } t \geq 1
\]
where $c_2(p)$ is bounded away from $0$ on every closed subinterval of $[0,p_0)$.
\end{thm}
\begin{remark}
For $p < p_0$ we obtain $\mathbb{E}_p\tau < \infty$ and consequently, by a mass transport argument (Lemma~\ref{lem:tau-V_duality} below) $\mathbb{E}_pV<\infty$, where $V$ is the total number of visits to the origin by cars. In \cite{DGJLS}, this was shown for $p < (256d^6e^2)^{-1} \asymp d^{-6}$. Our value of $p_0 \asymp d^{-2}$ improves on this bound. 
\end{remark}

\begin{remark}
The key ingredient in our proof is  Proposition~\ref{RW_exit_time}, which gives an exponential tail bound on the exit time of a random walk from a connected subset of $\mathbb{Z}^d$ of size $n$ that is uniform over all such subsets. This may be of independent interest, and our proof relies on the first moment bound of~\cite{LP} and spectral bounds in \cite[Cor.~6.9.5,~6.9.6]{LL}.
\end{remark}

\begin{remark}
We expect similar bounds to hold in the continuous-time diffusion-limited annihilating system (DLAS) with non-moving $B$ particles, wherein $A$ particles move as continuous-time simple, symmetric random walks. 
(An analogue to the ``busy subgraph lemma'' (Lemma~\ref{busy_subgraph}) would need to be proved for the continuous time model, which would no longer hold with a deterministic upper bound on the size of $H$. An upper bound with exponential tails would suffice, and should be true.) 
In that model, if both $A$ and $B$ particles perform continuous-time, simple, symmetric random walks with equal jump rates, then Bramson-Lebowitz \cite{BL} have shown that the density of $A$ particles decays to zero like $e^{-g_d(t)}$, with 
$$g_d(t) \asymp
\begin{cases}
\sqrt{t} & d=1\\
t/\log t & d=2\\
t & d\ge3,
\end{cases}
$$
when $p<1/2$ (in fact, their results are stronger than this). By a straightforward mass-transport argument, this is the same as the decay rate of the tail of the annihilation time distribution of a single $A$ particle. Therefore the above result indicates a clear difference between the case of non-moving $B$-particles and moving $B$-particles.
\end{remark}

Our result implies there is an asymmetry between stationary (spots) and mobile particles (cars) due to an argument of \cite{BL}, at least for $d=1$. When $d=1$, Theorem~\ref{thm: main_thm} states that when $p<1/2$ (so cars are in the minority), the distribution of the parking time of a car decays like $e^{-ct^{1/3}}$. In contrast, Theorem~\ref{thm: p>1/2} below says that when $p>1/2$ (spots are in the minority), the distribution of the parking time of a spot decays like $e^{-c\sqrt{t}}$.
\begin{thm}\label{thm: p>1/2}
Let $\sigma$ be the time at which the parking spot at the origin is parked in by a car (where $\sigma=0$ if there is no such spot). Suppose $d=1$ and $p>1/2$. Then there exist constants $a,b>0$ such that
$$
e^{-a\sqrt{t}}\le \mathbb{P}_p(\sigma>t)\le e^{-b\sqrt{t}} \qquad \text{ for all $t\ge 1$.}
$$
\end{thm}

\begin{proof}[Sketch proof of Theorem~\ref{thm: p>1/2}]

The lower bound on the probability follows from an argument similar to the proof of the lower bound in Theorem~\ref{thm: main_thm}, given in Section~\ref{sec:lower bound}. That is, the spot at the origin survives beyond time $t$ if all vertices within distance $C\sqrt{t}$ of the origin are initially spots, and all random walk trajectories associated to vertices beyond this distance do not visit the origin by time $t$. This event has probability at least $e^{-a\sqrt{t}}$.

The upper bound is more complicated and follows closely the argument of \cite[Sec.~7,~p.~363-371]{BL}; we give a (very) brief sketch here. If the spot of the origin survives beyond time $t$, then one can identify approximately $\sqrt{t}$ many random walks (car trajectories) that must avoid the origin. Since each trajectory has at least a uniformly positive probability $c$ to reach the origin in time $t$, one obtains the upper bound $(1-c)^{c\sqrt{t}}$, which is of order $e^{-b\sqrt{t}}$. To define these random walks, one labels cars $c_1, c_2, \dots$ as follows. Write $D(j), j \geq 0$ for the initial number of cars in $[0,j]$ minus the initial number of spots. Let $c_1$ be the car at the first vertex $j$ such that $D(k) \geq 0$ for all $k \geq j$. Define $c_r$ similarly: the car at the first vertex $j$ such that $D(k) \geq r-1$ for all $k \geq j$. (These cars are a.s.~well-defined because $p>1/2$, and for some $c,C>0$, at least $c\sqrt{t}$ labeled cars initially lie in $[0,C\sqrt{t}]$ with probability at least $1-e^{-c\sqrt{t}}$.) One can then pair cars and spots in the intervals between the $c_r$'s: strictly between the initial locations of $c_r$ and $c_{r+1}$, proceeding from left to right, each car can be paired uniquely to a spot initially to its right that is the first such unpaired spot. (Note that some cars between the origin and the initial location of $c_1$ remain unpaired.) Now, as the parking process evolves, if a paired car parks in a spot that is paired to different car (necessarily to its left), then this other car is re-paired with the spot that had been paired with the first car. We focus now on the cars labeled $\{c_r\}$. If $c_r$ parks in a paired spot, then we reassign the label $c_r$ to the car that is paired with that spot. In this way, we see that except for finitely many of the $c_r$'s (which may park in the unpaired spots between the origin and the initial location of $c_1$), if the spot of the origin survives for time $t$, then none of the $c_r$'s can reach the origin by time $t$. Any $c_r$ that begins within distance $C\sqrt{t}$ of the origin has probability at least $c>0$ to reach the origin within time $t$ because the trajectory of $c_r$ follows a random walk with only negative drift (each reassignment moves a car strictly to the left). Therefore we find that the probability that the spot of the origin survives for at least time $t$ is at most $(1-c)^{c\sqrt{t}} + e^{-c\sqrt{t}}$, as claimed.
\end{proof}


As a corollary to Theorem~\ref{thm: main_thm}, we also derive the small $p$ asymptotic for the expected total number of visits to the origin by cars. We say that a car \textit{visits} $x$ at time $t\ge1$ if it is unparked at time $t-1$ and moves to $x$ at time $t$ (it may or may not park at $x$). Each time a car visits the site $x$ is referred to as a \textit{visit}. Let $V_t$ denote the number of visits by all cars to the origin, $0$, through time $t$. Multiple cars may visit $0$ at the same time, and a car may visit $0$ at multiple times -- these are all counted as distinct visits. Let $V = \lim_{t\to\infty} V_t$ denote the number of visits to the origin for all time. The asymptotic behaviors of $\mathbb{E}_{1/2} V_t$ and $\mathbb{E}_p V$ as $t\to\infty$ and $p\uparrow 1/2$, respectively, are the subjects of~\cite{JJLS,PRS} on $\mathbb{Z}^d$. While it remains an open question whether $\mathbb{E}_pV<\infty$ for all $p<1/2$ and $d\ge 2$, here we give the small $p$ asymptotic for $\mathbb{E}_p V$ in every dimension.

\begin{thm}\label{thm:total_visits}
For all $p\in [0,1]$, we have $\mathbb{E}_pV \ge p+p^2$. As $p\to 0$ we have
$$
\mathbb{E}_p V = p + O\left(p^2 (\log p^{-1})^{(d+2)/d}\right),
$$
where the constant in the $O$ term depends on $d$.
\end{thm}
Essentially, this says that for small $p$ the origin is most likely never visited by a car, or else it is visited once by a car that is initially adjacent to the origin.

\section{Proof of Theorem~\ref{thm: main_thm}}

To formalize the model, we recall from \cite[Sec.~2]{DGJLS} that our space is
\[
\Omega = \left( \{-1,1\} \times (\mathbb{Z}^d)^\mathbb{N} \times [0,1]^\mathbb{N}\right)^{\mathbb{Z}^d},
\]
with probability measure $\mathbb{P}_p$ under which all coordinates are independent, and for each $v \in \mathbb{Z}^d$, the three components are distributed as follows. The first coordinate is a random variable with probability $p$ to be $1$ (if there is a car initially at $v$) and probability $1-p$ to be $-1$ (otherwise). The second is a simple symmetric random walk started at $v$, which is the path that an unparked car placed at $v$ will follow (this path continues past the parking time). The third is a sequence of i.i.d.~uniform $[0,1]$ random variables to break ties if multiple cars arrive at the same parking spot at the same time (the car that parks is the one with the largest value of uniform variable whose index corresponds to the present time).

\subsection{Lower bound}\label{sec:lower bound}
We first prove the easier inequality, the lower bound. The idea is to force a large box centered at the origin initially to contain only cars, and for the car initially at the origin to stay in this box until time $t$. Optimizing the size of this box gives the bound. For any integer $M>0$, let $A_M$ be the event that for all $v \in [-M,M]^d$, the vertex $v$ initially has a car. If $A_M$ occurs, and the car starting at $0$ does not leave $[-M,M]^d$ by time $t$, then $\tau > t$. By independence of the initial particle locations and the random walk trajectories, we have
\begin{align*}
\mathbb{P}_p(\tau > t) &\geq \mathbb{P}_p(A_M,~\text{car starting at } 0 \text{ stays in }[-M,M]^d \text{ through time }t) \nonumber \\
&= (1-p)^{(2M+1)^d} \mathbb{P}_p(X \text{ stays in }[-M,M]^d \text{ through time }t). \label{eq: lower_bound_1}
\end{align*}
Here, $X = (X_t)$ is the random walk started at $0$.
By \cite[Cor.~6.9.5,~6.9.6]{LL}, the second factor is bounded below by $c_4 \exp\left( -C_3t/M^2\right)$,
%
so we obtain
\begin{align*}
\mathbb{P}_p(\tau > t) &\geq (1-p)^{(2M+1)^d} \left( c_4 \exp\left( - C_3 \frac{t}{M^2}\right) \right) \\
&= c_4 \exp\left( -C_3 \frac{t}{M^2} + (2M+1)^d \log (1-p)\right) \\
&\geq c_4 \exp\left( - C_5 \left( \frac{t}{M^2} + M^d\right)\right),
\end{align*}
where $C_5$ depends on $p$ and $d$. Last, we set $M = \lfloor t^{1/(d+2)}\rfloor$ to obtain the bound
\[
c_4 \exp\left( - C_5 \left( \frac{t}{\lfloor t^{\frac{1}{d+2}}\rfloor^2} + \lfloor t^{\frac{1}{d+2}}\rfloor^d\right)\right) \geq c_6 \exp\left( - C_7 t^{\frac{d}{d+2}}\right).
\]
By increasing $C_1$ in the statement of Theorem~\ref{thm: main_thm}, this proves the lower bound.

\subsection{Upper bound}
For the upper bound, we refine the ``busy subgraph'' method introduced in \cite{DGJLS}. Specifically, we combine it with consequences of a spectral isoperimetric bound of Levine-Peres. To begin, let us recall the definition of a busy subgraph.

\begin{define}
We call a finite subgraph $H\subset \mathbb{Z}^d$ \emph{busy} if $H$ is connected and there are at least as many cars as spots initially on $H$. 
\end{define}
As stated in \cite[p.~2111]{DGJLS}, for $p<1/2$ and any fixed connected subgraph $H$ of $\mathbb{Z}^d$ with $j$ vertices,
\begin{equation}\label{eq: busy_bound}
\mathbb{P}_p(H \text{ is busy}) \leq \left(2\sqrt{p(1-p)}\right)^j.
\end{equation}
We therefore, from this point on, restrict to $p<1/2$.

Busy subgraphs are important because if $\tau > t > 0$, we can always construct one that contains the trajectory of the car of the origin until time $t$, as stated in \cite[Lem.~4.12]{DGJLS} and reproduced below. In its statement, $\mathbb{B}(v,2t)$ is the subgraph of $\mathbb{Z}^d$ induced by the set of vertices within $\ell^1$-distance $2t$ of $v$.
\begin{lemma}\label{busy_subgraph}
Let $t\ge 1$ and $v \in \mathbb{Z}^d$. For each $\omega \in \{\tau^{(v)} > t\}$, there is a busy subgraph $H = H(\omega)$ such that $H \subseteq \mathbb{B}(v,2t)$ and $H$ contains the trajectory of the car started at $v$ up to time $t$.
\end{lemma}
Following \cite[p.~2111]{DGJLS}, this lemma is used as follows. We first note that $|\mathbb{B}(0,2t)|\le (4t+1)^d$. Letting $X = (X_{t})_{t\ge 0}$ be a simple symmetric random walk trajectory on $\bZ^d$ with $X_{0}=0$, we then apply \eqref{eq: busy_bound}, Lemma~\ref{busy_subgraph}, and a union bound to obtain for integers $t \geq 1$ that
\begin{align}
& \mathbb{P}_p(\tau> t) \nonumber \\
&\le \sum_{j=1}^{(4t+1)^d} \sum_{\substack{H\text{ connected}\\ |H|=j, 0\in H}} \left(2\sqrt{p(1-p)}\right)^{j} \, \mathbb{P}\left( \text{$X_k\in H$ for all $0\le k \le t$} \right) \nonumber \\
&=  \sum_{j=1}^{(4t+1)^d} \sum_{\substack{H\text{ connected}\\ |H|=j, 0\in H}} \left(2\sqrt{p(1-p)}\right)^{j} \, \mathbb{P}(\mathfrak{t}_H> t). \label{eq: beginning_of_end}
\end{align}
Here, $\mathfrak{t}_H$ is the exit time of $H$ by $X$,
\[
\mathfrak{t}_H = \inf\{t \geq 0 : X_t \notin H\},
\]
and we abuse notation by identifying $H$ with its vertex set, so $|H|$ is the number of vertices in $H$ and $0\in H$ indicates that $0$ is a vertex of $H$.

At this point, our strategy diverges from that of \cite{DGJLS} as we can more finely control the distribution of the exit time using the following new proposition. Again, although we use the measure $\mathbb{P}_p$, the bound does not depend on $p$.
\begin{prop}\label{RW_exit_time}
Let $(X_t)_{t\ge 0}$ be a simple symmetric random walk on $\bZ^d$ with $X_0=0$. There exists $c_8>0$ such that
\[
\mathbb{P}_p(\mathfrak{t}_H > t) \le \sqrt{n} \exp\left(-c_8tn^{-\frac{2}{d}}\right)
\]
for all integers $t\geq 1$ and $n\ge 1$, and all subgraphs $H$ of $\mathbb{Z}^d$ with $|H|=n$.
\end{prop}
\noindent
Note that the order of the exponent in Proposition~\ref{RW_exit_time} is sharp by considering $H = [-M,M]^d$, which has $\mathbb{P}_p(\mathfrak{t}_H > t) \ge c_4 \exp\left( -C_3t/M^2\right)$ by \cite[Cor.~6.9.5,~6.9.6]{LL}, and taking $(2M+1)^d \leq n$.

We now show how to complete the proof of the upper bound in Theorem~\ref{thm: main_thm} given this proposition. Afterward, we will finish by proving the proposition. Applying it to \eqref{eq: beginning_of_end}, using the fact that there are at most $(2de)^j$ many connected subgraphs with $j$ vertices containing 0, and the fact that if $j=1$ then $\mathfrak{t}_H\le 1$, we obtain
\[
\mathbb{P}_p(\tau>t) \le \sum_{j=2}^{(4t+1)^d}  \left(4de\sqrt{p(1-p)}\right)^{j} \, \sqrt{j} \exp\left( -c_8t j^{-\frac{2}{d}}\right).
\]
(When $d=1$, the number of such subgraphs is at most $j+1$.) If $p < p_0$, the term $4de\sqrt{p(1-p)}$ is less than 1 (and for $d=1$ the corresponding term $2\sqrt{p(1-p)}$ is also less than 1 for $p<p_0$), and is bounded away from $1$ for $p$ in any closed subinterval of $[0,p_0)$, so we can bound the last expression above by
\begin{align*}
&(4de)^2 p\sqrt{t} \sum_{j=2}^{\left\lceil t^{\frac{d}{d+2}}\right\rceil}  \exp\left( -c_8t j^{-\frac{2}{d}}\right)\  +\ (4t+1)^{\frac{d}{2}}\sum_{j=\left\lceil t^{\frac{d}{d+2}}\right\rceil+1}^\infty  \left(4de\sqrt{p(1-p)}\right)^{j} \\
\le~& (4de)^2 p t^{\frac{3}{2}} \exp\left( -c_8t^\frac{d}{d+2} \right) \ +\  \frac{(4t+1)^{\frac{d}{2}}}{1- 4de\sqrt{p(1-p)}} \left(4de\sqrt{p(1-p)}\right)^{t^{\frac{d}{d+2}}+1}\\
\leq ~& p\cdot \exp\left( - c_2 t^{\frac{d}{d+2}}\right),
\end{align*}
for $c_2$ chosen sufficiently small and $t \geq C$, where $c_2,C$ can be chosen to hold for all $p$ in any fixed, closed subinterval of $[0,p_0)$. To handle $t \in [1,C)$, we write
\begin{align*}
P_p(\tau>t) 
&\leq \mathbb{P}_p(\tau > 1)\\
&\leq p\mathbb{P}_p(\text{at least one nonzero site within $\ell^1$-distance 2 of 0 is initially a car})\\
&\leq p (1-(1-p_0)^{4d^2})\\
&\leq pe^{-c_2 C^{d/(d+2)}}\\
&\leq pe^{-c_2 t^{d/(d+2)}}
\end{align*}
where $c_2$ is possibly chosen smaller (but still bounded away from 0 on any closed subinterval of $[0,p_0)$). This gives the upper bound in Theorem~\ref{thm: main_thm}.

We are therefore left to prove the proposition.
\begin{proof}[Proof of Proposition~\ref{RW_exit_time}]
For integers $n,t \geq 1$, let $H = H_n(t)$ be a subgraph of $\mathbb{Z}^d$ with $n$ vertices that maximizes $\mathbb{P}_p(\mathfrak{t}_H > t)$. Clearly $H$ contains the origin. We will view a random walk on $H$ as a Markov chain, so we define the matrix
\[
P_H \text{ with entries } (p(x,y))_{x,y \in H}.
\]
Here, $p(x,y)$ is the transition probability from $x$ to $y$, which for our simple symmetric random walk is $1/(2d)$ if $x$ and $y$ are neighbors, and 0 otherwise. Note that $P_H$ is symmetric but its rows do not necessarily sum to 1, since transitions out of $H$ are not represented. Write $P_H^t$ for the $t$-th product of $P_H$ with itself, with entries $(p_H^t(x,y))$. Then
\begin{equation}\label{eq: P_1}
\mathbb{P}_p(\mathfrak{t}_H > t) = \sum_{y \in H}p_H^t(0,y).
\end{equation}
Note that the right side actually equals $\max_{x \in H} \sum_{y \in H} p_H^t(x,y)$. Indeed, if $x \in H$, then
\[
\sum_{y \in H} p_H^t(x,y) = \mathbb{P}_p(\mathfrak{t}_H^x > t) = \mathbb{P}_p(\mathfrak{t}_{H^x} > t).
\]
Here, $\mathfrak{t}_H^x$ is the exit time from $H$ of a random walk started at $x$, and $H^x$ is the subgraph obtained from $H$ by shifting $x$ to the origin. Since $H^x$ is a subgraph of $\mathbb{Z}^d$ with $n$ vertices, maximality of $H$ implies that $\sum_{y \in H} p_H^t(x,y) \leq \sum_{y \in H} p_H^t(0,y)$. In conclusion, 
\[
\mathbb{P}_p(\mathfrak{t}_H > t) = \max_{x \in H} \sum_{y \in H} p_H^t(x,y).
\]
The term on the right is the maximal row-sum of $P_H^t$, which equals 
\[
\|P_H^t\|_\infty = \sup_{z \neq 0} \frac{\|P_H^tz\|_\infty}{\|z\|_\infty},
\]
the operator norm of the matrix $P_H^t$ considered as a map from $\ell^\infty$ to $\ell^\infty$. We conclude from this and \eqref{eq: P_1} that
\begin{equation}\label{eq: P_2}
\mathbb{P}_p(\mathfrak{t}_H > t) = \|P_H^t\|_\infty.
\end{equation}

We next relate the infinity-norm of $P_H^t$ to its eigenvalues. Letting $\alpha = \alpha_H$ be the largest eigenvalue of $P_H$, we claim that
\begin{equation}\label{eq: P_3}
\alpha^t \leq \mathbb{P}_p(\mathfrak{t}_H>t) \leq \sqrt{n}\alpha^t.
\end{equation}
The lower bound follows from \eqref{eq: P_2} by letting $v$ be an eigenvector for $P_H$ and noting that since $v$ is also an eigenvector for $P_H^t$ with eigenvalue $\alpha^t$,
\[
\|P_H^t\|_\infty \geq \frac{\|P_H^t v\|_\infty}{\|v\|_\infty} = \alpha^t.
\]
For the upper bound, since $\alpha^t = \sup_{z \neq 0} \|P_H^t z\|_2/\|z\|_2 =: \|P_H^t\|_2$, we can apply \eqref{eq: P_2} along with the estimate
\[
\|P_H^t z\|_\infty \leq \|P_H^tz\|_2 \leq \|P_H^t\|_2 \|z\|_2 \leq \sqrt{n}\|P_H^t\|_2 \|z\|_\infty
\]
to obtain $\|P_H^t\|_\infty \leq \sqrt{n}\|P_H^t\|_2 = \sqrt{n}\alpha^t$.

After \eqref{eq: P_3}, we need a theorem of Levine--Peres \cite[Thm.~1.2]{LP} on the expected exit time of domains by random walk. It is a form of a spectral isoperimetric inequality, and implies that there exists $C_9>0$ such that
\begin{equation}\label{eq: LP}
\sup_{K \subset \mathbb{Z}^d, |K|=n} \mathbb{E}_p\mathfrak{t}_K \leq C_9 \mathbb{E}_p \mathfrak{t}_{B_n} \text{ for all } n \geq 1,
\end{equation}
where $B_n$ is the ``lattice ball'' of cardinality $n$. This is defined as the subgraph induced by the ``first $n$ points in an ordering of points in $\mathbb{Z}^d$ according to increasing distance from the origin.'' (Actually the result is more precise, but we only need the existence of $C_9$.) Note that $B_n \subset [-C_{10} n^{1/d},C_{10}n^{1/d}]^d$ for some constant $C_{10}>0$ and therefore 
\[
\mathbb{E}_p \mathfrak{t}_{B_n} \leq \mathbb{E}_p \mathfrak{t}_{[-C_{10}n^{1/d}, C_{10}n^{1/d}]^d}.
\]
By integrating the upper bound of \cite[Cor.~6.9.6]{LL} we obtain the standard estimate that the right side of the above inequality is bounded above by $C_{11}n^{2/d}$. Putting this in \eqref{eq: LP}, we get
%
%
\begin{equation}\label{eq: our_LP}
\sup_{K \subset \mathbb{Z}^d, |K| = n} \mathbb{E}_p \mathfrak{t}_K \leq C_{11} n^{\frac{2}{d}}.
\end{equation}

Last, we use the above tools to complete the proof. Using the lower bound of \eqref{eq: P_3} and inequality \eqref{eq: our_LP}, we obtain
\[
C_{11}n^{\frac{2}{d}} \geq \sum_{t \geq 0} \mathbb{P}_p(\mathfrak{t}_H > t) \geq \sum_{t\geq 0} \alpha^t = \frac{1}{1-\alpha},
\]
and solving this for $\alpha$, we see that
\[
\alpha \leq 1- \frac{1}{C_{11}n^{\frac{2}{d}}}.
\]
Plugging this into the upper bound of \eqref{eq: P_3}, we finish with
\begin{align*}
\mathbb{P}_p(\mathfrak{t}_H \geq t) \leq \sqrt{n} \alpha^t \leq \sqrt{n} \left( 1-\frac{1}{C_{11} n^{\frac{2}{d}}}\right)^t &= \sqrt{n} \exp\left( t \log \left( 1 - \frac{1}{C_{11} n^{\frac{2}{d}}}\right)\right) \\
&\leq \sqrt{n} \exp\left( - c_{12}tn^{-\frac{2}{d}}\right),
\end{align*}
which is the bound of Proposition~\ref{RW_exit_time}.
\end{proof}

\section{Proof of Theorem~\ref{thm:total_visits}}
Our proof uses Theorem~\ref{thm: main_thm} in combination with the following lemma, which appears with a small error in~\cite[Prop.~4.10]{DGJLS}.  We provide the corrected (short) proof for completeness.

\begin{lemma}\label{lem:tau-V_duality} For all $t\ge 1$ we have
$$
\mathbb{E}_pV_t = \sum_{s=1}^t \mathbb{P}_p(\tau\ge s).
$$
\end{lemma}
\begin{proof}
For each $x,y\in \mathbb{Z}^d$ and integers $s\ge 1$ define
$$
Z_s(x,y) = \mathbf{1}\{\text{a car is at $x$ initially and visits $y$ at time $s$}\}.
$$
Then, for $s\ge 1$,
$$
\sum_{y\in \mathbb{Z}^d} Z_s(y,0) = \#\{\text{cars visiting $0$ at time $s$}\} = V_s - V_{s-1}.
$$
Also,
$$
\sum_{y\in \mathbb{Z}^d} Z_s(0,y) = \mathbf{1}\{\text{a car starts at $0$ and is unparked at time $s-1$}\} = \{\tau\ge s\}.
$$
Taking expectations and using the fact that translations are measure preserving for $\mathbb{P}_p$, we have
\begin{align*}
\mathbb{E}_p (V_s - V_{s-1}) &= \sum_{y\in \mathbb{Z}^d} \mathbb{E}_p Z_s(y,0) =  \sum_{y\in \mathbb{Z}^d} \mathbb{E}_p Z_s(0,-y) =  \sum_{y\in \mathbb{Z}^d} \mathbb{E}_p Z_s(0,y) = \mathbb{P}_p(\tau\ge s).
\end{align*}
Summing from $s=1$ to $t$ and noting that $V_0 = 0$ finishes the proof.
\end{proof}

Observe  that $\mathbb{P}_p(\tau\ge 1) = p$.  The lower bound in Theorem~\ref{thm:total_visits} will therefore  follow from Lemma~\ref{lem:tau-V_duality} once we show that $\mathbb{P}_p(\tau\ge 2)\ge p^2$. Let $y$ be a neighbor of $0$ in $\mathbb{Z}^d$, denoted $y\sim 0$, and let $A_y$ be the event that there is a car initially at the origin and it moves to $y$ at time $1$. Let $B_y$ be the event that there is a car initially at $y$. Observe that $A_y$ and $A_z$ are disjoint whenever $y\ne z$, and $A_y$ and $B_y$  are independent, so
\begin{align*}
\mathbb{P}_p(\tau\ge  2) \ge \mathbb{P}_p\left(\bigcup_{y:  y\sim 0}  (A_y\cap B_y)\right) =  \sum_{y:y\sim0}\mathbb{P}_p(A_y \cap B_y) = \sum_{y:y\sim0} (p/2d)(p)  = p^2.
\end{align*}

For the upper bound, fix $p^*\in [0,p_0)$, and for each $p\in[0,p^*]$ let $c_2(p)$ be the constant that appears in the upper bound in Theorem~\ref{thm: main_thm}. Let $c_{13} =  \inf_{p\in[0,p^*]} c_2(p)>0$, so that for all $p\in[0,p^*]$ and all $t\ge 1$ we have
\begin{equation}\label{tau tail}
\mathbb{P}_p(\tau>t)\le p\exp\left(-c_{13} t^{\frac{d}{d+2}}\right).
\end{equation}
Let 
$$
k = \left\lceil \left(-\frac{2}{c_{13}} \log p \right)^{(d+2)/d}  \right\rceil.
$$
Since $\mathbb{P}_p(\tau>0) = p$ and 
$$
\mathbb{P}_p(\tau>1) \le \mathbb{P}_p(\text{initially, there is a car at $0$ and another car within two steps of $0$}) \le 4d^2p^2,
$$
and $\mathbb{P}_p(\tau>t)$ is decreasing in $t$, it follows that
\begin{equation}\label{small t bound}
\sum_{t=0}^k\mathbb{P}_p(\tau>t) \le p + k(4d^2p^2)\le p + C_{14}p^2 (\log(p^{-1}))^{(d+2)/d}
\end{equation}
for a sufficiently large constant $C_{14}$.

Applying the bound~\eqref{tau tail} and using that $e^{-c_{13}t^{\frac{d}{d+2}}}$ is decreasing in $t$, we have
\begin{align}
\nonumber \sum_{t=k+1}^\infty \mathbb{P}_p(\tau>t) &\le \sum_{t=k+1}^\infty p\exp\left(-c_{13} t^{\frac{d}{d+2}}\right)\\
\nonumber &\le p\int_k^\infty e^{-c_{13}t^{\frac{d}{d+2}}}\, dt\\
\nonumber &= p \int_{c_{13}k^{d/(d+2)}}^\infty \frac{d+2}{d(c_{13})^{(d+2)/d}} u^{2/d} e^{-u}\, du\\
\label{integral bound} &= C_{15}\,p \int_{c_{13}k^{d/(d+2)}}^\infty u^{2/d} e^{-u}\, du.
\end{align}
Now assume $p$ is such that $-2\log p \ge 2$, and note that $c_{13}k^{d/(d+2)}\ge -2\log p$. For $d=1$, by evaluating the integral, \eqref{integral bound} is at most
\begin{align*}
5C_{15}\, p(c_{13}k^{1/3})^2 \exp(-c_{13}k^{1/3}).
\end{align*}
Plugging in our choice of $k$ and noting that $u^2 e^{-u}$ is a decreasing function for $u\ge 2$, the last expression is at most
$$
20 C_{15}\,p^3(\log p)^2 \le C_{16} p^2 (\log(p^{-1}))^{3}
$$
for a large enough constant $C_{16}$. Combining this with~\eqref{tau tail} and~\eqref{small t bound} finishes the proof for $d=1$.  For $d\ge 2$, since $u^{2/d}\le u$ for $u\ge 1$ we have that~\eqref{integral bound} is at most
$$
C_{15}\,p\int_{c_{13}k^{d/(d+2)}}^\infty u e^{-u}\, du\le 2C_{15}\,  p (c_{13}k^{d/(d+2)})\exp(-c_{13}k^{d/(d+2)}),
$$
and plugging in our choice of $k$ and noting that $ue^{-u}$ is decreasing for $u\ge 1$ gives an upper bound of
$$
4 C_{15}\, p^3 \log p^{-1} \le C_{17} p^2 (\log p^{-1})^{(d+2)/d}
$$
for a large  enough constant $C_{17}$. Combining this with~\eqref{tau tail} and~\eqref{small t bound} finishes the proof for $d\ge 2$. \qed

\bigskip
\noindent
{\bf Acknowledgements.} The research of M.~D. is supported by an NSF CAREER grant. D.~S. is partially supported by the NSF grant CCF--1740761. H.~L. is partially supported by the NSF grant DMS--2010035. H.~L. thanks Yuval Peres for pointing out the reference~\cite{LP}.

\end{document}